\documentclass[12pt,reqno]{amsart}
\usepackage{amssymb}
\usepackage[curve,matrix,arrow]{xy}
\newtheorem{thm}{Theorem}[section]
\newtheorem{lemma}[thm]{Lemma}
\newtheorem{cor}[thm]{Corollary}
\newtheorem{prop}[thm]{Proposition}

\newtheorem{defi}[thm]{Definition}
\theoremstyle{definition}
\newtheorem{rem}[thm]{Remark}

\def\CC{{\mathcal{C}}}
\def\CD{{\mathcal{D}}}

\def\CF{{\mathcal{F}}}

\def\CI{{\mathcal{I}}}

\def\CU{{\mathcal{U}}}

\def\CY{{\mathcal{Y}}}

\def\Fu{{\mathfrak{u}}}
\def\bZ{{\mathbb Z}}

\def\res{\operatorname{res}\nolimits}

\def\Rad{\operatorname{Rad}\nolimits}

\def\modg{\operatorname{{\bf mod}(\text{$kG$})}\nolimits}

\def\stmodg{\operatorname{{\bf stmod}(\text{$kG$})}\nolimits}
\def\stmod{\operatorname{{\bf stmod}}\nolimits}
\def\Stmodg{\operatorname{{\bf StMod}(\text{$kG$})}\nolimits}

\def\HHH{\operatorname{H}\nolimits}
\def\Hom{\operatorname{Hom}\nolimits}

\def\HH#1#2#3{\HHH^{#1}(#2,#3)}

\def\homb{\operatorname{\underline{Hom}}\nolimits}

\def\Ext{\operatorname{Ext}\nolimits}

\def\hgs{\HH{*}{G}{k}}

\def\bfp{{\mathbb F}_p}

\title{Blocks and support varieties}
\author{Jon F. Carlson and Jeremy Rickard}

\begin{document}
\maketitle

\section{Introduction}

Block decompositions of module categories are well-known and much
studied, especially in modular representation theory. When considering
cohomological questions, it is often more convenient to work in the
stable module category, but this makes little difference to the block
theory: one simply loses the simple blocks for which all modules are
projective. The theory of varieties for modules for finite groups
gives a rich supply of interesting thick subcategories of the stable
module category. There are block decompositions of these arising from
the usual blocks of the group algebra, but it turns out that in
general the blocks break up even further. In this paper we study this
phenomenon, particularly in the case of the thick subcategory
determined by a single line in the maximal ideal spectrum $V_G(k)$ of
the cohomology algebra. We give simple examples where calculations can
easily be done, and use a theorem of Benson to reduce the general case
to examples of this kind.

Finally, we describe how the theory of block varieties developed by
Linckelmann seems to shed more light on these phenomena.

\section{Block decompositions of categories}

If $\CC$ is an additive category, then by an \textbf{additive
  subcategory} of $\CC$ we mean a full subcategory $\CC'$ such that if
$X$ is a finite coproduct, in $\CC$, of objects of $\CC'$, then $X$ is
in $\CC'$. In particular, $\CC'$ is also an additive category, and
contains all objects of $\CC$ isomorphic to objects of $\CC'$.

Let $\CC$ be an additive category. We say that $\CC$ is the
\textbf{direct sum} of a family $\{\CC_i:i\in I\}$ of additive
subcategories of $\CC$, and write
$$\CC=\bigoplus_{i\in I}\CC_i$$
if every object $X$ of $\CC$ can be expressed as a coproduct
$$X=\bigoplus_{i\in I}X_i$$
with $X_i$ an object of $\CC_i$, and
$$\Hom(X,Y)=0=\Hom(Y,X)$$
whenever $X$ is an object of $\CC_i$ and $Y$ is a coproduct of objects
of $\{\CC_j:j\neq i\}$.

It follows that if $X=\bigoplus X_i$ and $Y=\bigoplus Y_i$, with
$X_i,Y_i\in\CC_i$, then
$$\Hom(X,Y)=\prod_{i\in I}\Hom(X_i,Y_i).$$

It is easy to see that the projection $X\mapsto X_j$ from $\CC$
to $\CC_i$ is functorial for each $i\in I$, and is both left and right
adjoint to the inclusion functor $\CC_j\rightarrow\CC$. It follows
that the direct sum decomposition of $X$ is unique up to natural
isomorphism, and that the subcategories $\CC_i$ are closed under all
limits and colimits in $\CC$. In particular, they are closed under
arbitrary products and coproducts and under taking direct summands.

If $\CC$ is an abelian category then $\CC_i$ is abelian, closed under
taking extensions, subobjects and quotients, and the inclusion and
projection functors are exact.

If $\CC$ is a triangulated category, then if $\CC_i$ is closed under
the shift functor, it is triangulated, with the inclusion and
projection functors exact. In this paper we shall mostly be
considering the case where $\CC$ is a triangulated subcategory of the
stable module category $\stmod(kG)$ of a finite group algebra, and in
this case it follows automatically that each $\CC_i$ is closed under
the shift functor, by Tate duality.
$$\homb(M,N)\cong(\homb(N,\Omega(M))^*$$
for finitely generated modules $M$ and $N$, so in particular
$$\homb(M,\Omega(M))\neq0\neq\homb(\Omega^{-1}(M),M),$$
and hence if $M$ is an indecomposable object of $\CC_i$ then
$\Omega(M)$ and $\Omega^{-1}(M)$ must also be in $\CC_i$.

By a \textbf{block decomposition} of an additive category $\CC$, we
mean a direct sum decomposition $\CC=\bigoplus_{i\in I}\CC_i$ such
that the subcategories $\CC_i$ are nonzero and do not themselves have
nontrivial direct sum decompositions. We call the subcategories
$\CC_i$ \textbf{blocks} of $\CC$.

If $\CC$ has a block decomposition $\CC=\bigoplus_{i\in I}\CC_i$ and
another direct sum decomposition $\CC=\bigoplus_{i'\in I'}\CC_{i'}$
into non-zero direct summands, then for each $i\in I$ the projection
of $\CC_i$ onto $\CC_{i'}$ is nonzero for a unique $i'\in I'$, or else
$\CC_i$ would have a nontrivial direct sum decomposition. Hence each
$\CC_{i'}$ is a direct sum of blocks, and in particular the block
decomposition of $\CC$ is unique.

For example, if $kG$ is a finite group algebra then it is easy to
check that the module category $\modg$ has a block decomposition into
the module categories of the blocks, in the usual sense, of the group
algebra. Also the stable module category $\stmodg$ has a block
decomposition into the stable module categories of the nonsimple
blocks of $kG$. We give the easy proof of this later in this section.

We shall now describe the block decomposition of an arbitrary thick
subcategory of a stable module category.

\begin{defi}
  Let $\CC$ be a thick subcategory of the stable module category
  $\stmodg$ of a finite group, and let $\CI$ be the class of
  indecomposable objects of $\CC$. Define $\sim$ to be the smallest
  equivalence relation on $\CI$ such that $M\sim N$ whenever
  $\homb(M,N)\neq 0$.
\end{defi} 

Thus, for $M,N\in\CI$, $M\sim N$ if and only if there exist objects
$M=L_0,\dots,L_n=N$ of $\CI$ such that for every $i = 1,\dots,n$,
either $\homb(L_{i-1},L_i)\neq0$ or $\homb(L_i,L_{i-1})\neq0$. By the
remark on Tate duality above, it follows that $M\sim\Omega(M)$ for
every $M\in\CI$, and hence $M\sim N$ if
$$0\neq\widehat{\Ext}_{kG}^i(M,N)\cong\homb(\Omega^i(M),N)$$
for any $i\in\bZ$.

If $I$ is the set of equivalence classes, and we define $\CC_i$ to be
the full subcategory of $\CC$ consisting of direct sums of objects of
$i\in I$, then it is easy to see that $\CC=\bigoplus_{i\in I}\CC_i$ is
a block decomposition of $\CC$. We shall call the blocks $\CC_i$
\textbf{ext-blocks} of $\CC$ to distinguish them from the blocks of
the group algebra. In fact, the main aim of this paper is to study the
relationship between the two notions of block.

We shall be studying in detail the ext-blocks of subcategories of
$\stmodg$ determined by varieties. If $V$ is a closed homogeneous
subvariety of the maximal ideal spectrum $V_G(k)$ of $\hgs$, then we
denote by $\CC_V$ the full subcategory of $\stmodg$ consisting of the
modules $M$ whose variety $V_G(M)$ is contained in $V$. This is a
thick subcategory of $\stmodg$. We denote by $\sim_V$ the equivalence
relation described above on the class of indecomposable objects of
$\CC_V$.

\begin{prop} \label{extblock-properties} Suppose that $V$ is a closed
  homogeneous subvariety of $V_G(k)$. Then we have the following.
\begin{enumerate}
\item The number of ext-blocks of $\CC_V$ is finite.
\item If $\CC_V =\oplus_{i\in I}\CC_i$ is a direct sum decomposition,
  then any ext-block of $\CC_V$ is contained in some $\CC_i$.
\item If $M$ and $N$ are nonprojective indecomposable modules in the
  same ext-block of $\CC_V$, then $M$ and $N$ are in the same ordinary
  block of $kG$.
\item If $M$ is a nonprojective module in $\CC_V$ and if $M$ lies in a
  block $B$ of $kG$ with defect group $D$, then $V_G(M) \cap
  \res_{G,D}^*(V_D(k)) \neq \{0\}$.
\end{enumerate}
\end{prop}

\begin{proof}
  Let $L$ be a finitely generated $kG$-module with the property that
  $V_G(L)=V$. The fact that such a module exists is a standard
  property of support varieties for finite groups as in
  \cite{car-connected}.  Suppose that $M$ is a nonprojective
  indecomposable module in $\CC_V$. Then there exists an irreducible
  $kG$-module $S$ such that $\homb(M \otimes L, S) \neq0$. We
  know this from the tensor product theorem for support varieties
  which tells us that $M \otimes L$ is not projective since the
  varieties of $M$ and $L$ do not intersect trivially. Then we have
  that $\homb(M, L^* \otimes S)\neq0$ and hence there is some
  indecomposable component $U$ of $L^* \otimes S$ such that $M\sim_V
  U$.  We know that there are only a finite number of simple modules
  $S$ and only a finite number of components of $L^* \otimes S$ for
  any $S$.  Consequently, there are only a finite number of
  equivalence classes for the relation $\sim_V$, and hence only a
  finite number of ext-blocks.

  Statement $(2)$ repeats a general property of blocks proved above,
  and $(3)$ follows, since there is clearly a direct sum decomposition
  of $\CC_V$ according to the ordinary blocks of $kG$. 

  To prove statement $(4)$, we just need to recall that every module
  $N$ in $B$ is a direct summand of a module that is induced from a
  $kD$-module.  Therefore, $V_G(N)$ is contained in $\res_{G,D}^*(V_D(k))$.
\end{proof} 

\begin{rem}
  We should point out that the statement $(1)$ of the proposition is
  in contrast to the fact, shown in \cite{BCR3}, that the category
  $\CC_V$ may have an infinite number of mutually orthogonal thick
  subcategories. By this we mean that there may be an infinite number
  of thick subcategories such that if $M$ is a module in one and $N$
  is in another then $\widehat{\Ext}_{kG}^*(M,N)=0$.  However, $\CC_V$
  is not the direct sum of these subcategories, as they do not contain
  all indecomposable objects of $\CC_V$. Indeed, in the examples
  considered in \cite{BCR3} it can be shown that $\CC_V$ has only one
  ext-block.
\end{rem}

Notice that Proposition \ref{extblock-properties}(3) says that the
ext-blocks are a refinement of the ordinary blocks of $kG$. This
refinement can be seen another way.

\begin{prop}
  Suppose that $V$ and $V'$ are closed homogeneous subvarieties of
  $V_G(k)$ with $V \subseteq V'$, and let $M$ and $N$ be
  indecomposable modules in $\CC_V$.  Then if $M\sim_V N$, then also $M
  \sim_{V'} N$. Moreover, if $V = V_G(k)$, then the ext-blocks of
  $\CC_V=\stmodg$ are precisely the blocks of $kG$ which have defect
  greater than zero.
\end{prop}

\begin{proof}
  The first statement is obvious from the definition. The second is a
  well known fact about blocks. That is, if $V = V_G(k)$ and if $M$
  and $N$ are nonprojective modules in the same block of $kG$, then
  clearly $M \sim_V S$ and $N \sim_V S'$ for some nonprojective simple
  modules $S$ and $S'$ in the block. But then there are simple modules
  $S=S_0,\dots,S_n=S'$ in the block with $\Ext^1(S_i,S_j)\neq0$, so
  $S\sim_V S'$.
\end{proof}

We end this section with some remarks on how things change if we
consider the stable category $\Stmodg$ of arbitrary (not necessarily
finitely generated) modules. For many of the most familiar
subcategories, it turns out that the ext-block structure is the same
as in the finitely generated case.

To make this precise, for a thick subcategory $\CC$ of $\stmodg$, let
$\CC^{\oplus}$ be the localizing subcategory of $\Stmodg$ generated by
$\CC$; i.e., the smallest triangulated subcategory of $\Stmodg$
that contains $\CC$ and is closed under arbitrary coproducts.

\begin{prop}
  Suppose $\CC$ is a thick subcategory of $\stmodg$ that has a block
  decomposition $\CC=\bigoplus_{i\in I}\CC_i$. Then
  $\CC^{\oplus}=\bigoplus_{i\in I}\CC_i^{\oplus}$ is a block
  decomposition. In particular, there is a natural bijection between
  the blocks of $\CC$ and the blocks of $\CC^{\oplus}$.
\end{prop}

\begin{proof}
  Let $X$ be an object of $\CC_i^{\oplus}$ and let $Y=\bigoplus_{j\neq
    i}Y_j$ be a coproduct of objects of $\{\CC_j:j\neq i\}$. For any
  object $X'$ of $\CC_i$ and any object $Y_j'$ of $\CC_j$, where
  \mbox{$i\neq j$}, $\homb(X',Y_j')=0$. Since $X'$ is a compact object of
  $\Stmodg$, the functor $\homb(X',-)$ preserves arbitrary coproducts,
  and so $\homb(X',Y)=0$. Since the class of objects with no nonzero
  maps to $Y$ is a localizing subcategory of $\Stmodg$ that contains
  $\CC_i$, it contains $\CC_i^{\oplus}$, and hence $\homb(X,Y)=0$. A
  similar proof shows that $\homb(Y,X)=0$. Since $\bigoplus_{i\in
    I}\CC_i^{\oplus}$ is a localizing subcategory of $\CC^{\oplus}$
  that contains $\CC$, it must be the whole of $\CC^{\oplus}$. So
  $\CC^{\oplus}=\bigoplus_{i\in I}\CC_i^{\oplus}$ is a direct sum
  decomposition.
  
  It remains to show that $\CC_i^{\oplus}$ has no nontrivial direct
  sum decomposition. Suppose that $\CC_i^{\oplus}=\CD\oplus\CD'$. Then
  since $\CC_i$ has no nontrivial direct sum decomposition, either
  $\CD$ or $\CD'$ must contain every object of $\CC_i$. But every
  object of $\CC_i^{\oplus}$ has a nonzero map from some object of
  $\CC_i$, so either $\CD$ or $\CD'$ contains all objects of
  $\CC_i^{\oplus}$.
\end{proof}
  
\section{Fixed lines in the variety of a normal elementary abelian subgroup}
\label{sec-fixed}

In this section, we shall show that there is one important situation
where the ext-blocks of the category $\CC_V$ coincide with the
ordinary blocks.

If $V$ is a closed subvariety of $V_G(k)$ then we say that $V$ is
\textbf{minimally supported} on an elementary abelian subgroup $E$ of
$G$ if $E$ is a minimal elementary abelian subgroup such that $V
\subseteq \res_{G,E}^*(V_E(k))$. We know that if $V$ is an irreducible
subvariety, then $V$ is minimally supported on some $E$, which is
unique up to conjugacy in $G$.

\begin{thm} \label{main-thm} Suppose that the finite group $G$ has a
  normal elementary abelian subgroup $E$.  Let $V'$ be a line in
  $V_E(k)$; i.e., an irreducible linear subspace of $V_E(k)$ of
  dimension one. Let $V = \res_{G,E}^*(V')$, and assume that $V$ and
  $V'$ have the properties that
\begin{enumerate}
\item $V$ is minimally supported on $E$, and
\item $V'$ is stable under the action of $G/C_G(E)$. 
\end{enumerate}
 Then the ext-blocks of $\CC_V$ coincide with the ordinary blocks.
\end{thm}

For the remainder of this section we shall assume the hypotheses and
notation of the theorem.  Let $E = \langle x_1, \dots, x_n \rangle$
have rank $n$.

\begin{lemma} \label{lem1} $G/C_G(E)$ is a cyclic group of order prime
  to $p$, and acts on $V'$ by a linear character $\chi$.
\end{lemma}

\begin{proof}
  Suppose that $E$ has rank $n$ and $\alpha = (\alpha_1, \dots,
  \alpha_n)$ is a nonzero point of the rank variety of $E$ corresponding
  to the line $V'$. Then the element 
$$u_{\alpha} = 1 + \sum_{i = 1}^n \alpha_i(x_i-1) \in kE$$ 
has the property that for any $kE$-module $M$, $V' \in V_E(M)$ if and
only if the restriction $M_{\langle u_{\alpha} \rangle}$ is not free
as a $k\langle u_{\alpha} \rangle$-module.  By condition $(1)$, the
elements $\alpha_1, \dots, \alpha_n$ must be linearly independent over
the prime subfield $\bfp \subseteq k$, since, if there existed an
$\bfp$-dependence relation involving the elements $\alpha_1, \dots,
\alpha_n$, then we could find some proper linear subspace of $k$,
defined over $\bfp$, that contained $\alpha$.  But then this linear
subspace would be $\res_{E,F}^*(V_F(k))$ for some proper subgroup
$F\subseteq E$, contradicting condition $(1)$.  It follows that, since
elements of $G/C_G(E)$ act on $V_E(k)$ by an $\bfp$-linear
transformation, no nontrivial element of $G/C_G(E)$ can fix the line
$V'$ pointwise.  Consequently, the action on $V'$ gives us a faithful
representation $\xymatrix{\chi: G/C_G(E) \ar[r] & GL(1,k)}$ and so
$G/C_G(E)$ must be cyclic.
\end{proof}

The primary tool that we need is the following.

\begin{prop} \label{exist-shifted} There exists an element $\Fu \in
  \Rad(kE)$, $\Fu \notin \Rad^2(kE)$ having the following properties,
  where we denote by $U = \langle u \rangle$ the subgroup of the group
  of units of $kE$ generated by the element $u = \Fu +1$, so $U$ is
  cyclic of order $p$.
\begin{enumerate}
\item For $x \in G/C_G(E)$, $x\Fu x^{-1} = \chi(x)\Fu$, for $\chi$ as 
in Lemma \ref{lem1}.
\item If $M$ is a $kE$-module, then $V \subseteq V_E(M)$ if and only 
if $M_U$ is not a free $kU$-module.
\item If $M$ is a $kG$-module, then $V' \subseteq V_E(M)$ if and only 
if $M_U$ is not a free $kU$-module, where $V = \res_{G,E}^*(V')$. 
\end{enumerate}
\end{prop}

\begin{proof} Corresponding to the line $V$ in $V_E(k)$, there is a
  cyclic shifted subgroup $U = \langle u_{\alpha} \rangle$, such that
  $\alpha = (\alpha_1, \dots, \alpha_n) \in k^n$ and
$$u_{\alpha} = 1 + \sum_{i = 1}^n \alpha_i (x_i -1)$$
as in the previous proof.  For any $kG$-module $M$, $V' \subseteq
V_E(M)$ if and only if the restriction of $M$ to $U$ is not a free
$kU$-module.  Moreover, if we let $u = u_{\alpha} +w$ where $w \in
\Rad^2(kG)$, then the subgroup generated by $u$ has the same property
(see \cite{car-sec6}).

The fact that $V$ is invariant under the action of $G/C_G(E)$, implies
that $u_{\alpha}-1$ must be an eigenvector in the space
$\Rad(kE)/\Rad^2(kE)$ for the action of $G/C_G(E) = \langle x \rangle$
with the eigenvalue $\chi(x)$ for the element $x$. Because $G/C_G(E)$
has order prime to $p$, there is an element $\Fu$ in $\Rad(kG)$ where
$x$ acts with eigenvalue $\chi(x)$ and with the property that
$$
\Fu \ \equiv \ u_{\alpha} -1  \qquad \text{mod} \ \Rad^2(kG).
$$
Taking $u = 1 + \Fu$ and $U = \langle u \rangle$ proves parts $(1)$
and $(2)$. Part $(3)$ follows from the fact that a $kG$-module has the
property that $V \subseteq V_G(M)$ if and only if the restriction of
$M$ to $E$ has the property that $V' \subset V_E(M_E)$
\cite{alp-evens}.
\end{proof}

Let $X = kG/\Fu kG$ be the quotient module. We claim that $X \cong
(kE/\Fu kE)^{\uparrow G}$. This is because clearly 
$$
(kE/\Fu kE)^{\uparrow G} \cong
kG/(\Fu kE)^{\uparrow G}),
$$
and 
$$
(\Fu kE)^{\uparrow G}\cong
kG\otimes_{kE}\Fu kE \cong kG\Fu\otimes_{kE}kE 
\cong kG\Fu,
$$
which is the same as $\Fu kG$ by Proposition~\ref{exist-shifted}(i).

Also $V_G(X) = V$. This is true by a rank variety argument. That is,
$kE/\Fu kE$ is an indecomposable periodic module, and hence its rank
variety is a single line, which must be the line through $\alpha$.

Now we prove the following.

\begin{lemma} \label{lem2} Suppose that $M$ is a nonprojective
  $kG$-module in $\CC_V$. Then $\Ext_{kG}^n(M,X)$ is nonzero for all
  $n$.
\end{lemma}

\begin{proof}
This follows from the fact that $X$ is induced from a 
$p$-subgroup: 
$$\Ext_{kG}^n(M,X)\cong \Ext^n_{kG}(M, (kE/\Fu kE)^{\uparrow G})\cong
\Ext^n_{kE}(M,kE/\Fu kE)\neq0$$
since the varieties of $M$ and $kE/\Fu kE$ both contain $V'$.
\end{proof}

For each indecomposable projective summand $P$ of $kG$, $X$ has a
corresponding summand $P/\Fu P$, which is indecomposable (since it has
a simple top) and non-projective (since its restriction to $kE$ is a
direct sum of copies of $kE/\Fu kE$).  At this point we fix a
$p$-block $B$ of $G$, let $\{P_1,\dots,P_t\}$ be a complete set of
representatives of the isomorphism classes of the projective
indecomposable $kG$-modules in $B$, and let $X_i=P_i/\Fu P_i$ for
$i\in\{1,\dots,t\}$, so $\{X_1,\dots,X_t\}$ is a complete set of
representatives of the isomorphism classes of indecomposable summands
of $X$ in the block $B$.

So $X_1,\dots,X_t$ are in $\CC_V\cap B$ (which is therefore nonzero),
and Lemma \ref{lem2} implies that if $M$ is any other module in $\CC_V
\cap B$, then there is some $i\in\{1, \dots, t\}$ such that $M \sim_V
X_i$.  Consequently, in order to show that the modules in $\CC_V \cap
B$ are all in the same ext-block, and to prove Theorem~\ref{main-thm},
it remains to show that $X_i, \sim_V X_j$ for all $i$ and $j$. The
first step in this direction is the following. Recall that $\chi$ is
the character having $C_G(E)$ as its kernel such that $g \Fu g^{-1} =
\chi(g) \Fu$.  Let $\CY_i$ be the $kG$-module of dimension one which
affords the character $\chi^i$.

\begin{lemma} \label{lem3}
For any $i$ and $j$, we have that $\CY_i \otimes X_j \sim_V X_j$.
\end{lemma}

\begin{proof}
  First notice that if we fix a nonzero element $y\in\CY_i$, the map
$$
\xymatrix{
\mu: \CY_1 \otimes P_j  \ar[r] & P_j
}
$$ 
given by $\mu(y \otimes x) = \Fu x$ is a $kG$-module homomorphism,
since for $g \in G$,
$$
\mu(g(y\otimes x)) = \mu(\chi(g)y\otimes gx) = \chi(g)\Fu gx = 
g\Fu g^{-1}gx = g\mu(y \otimes x).
$$
The cokernel of $\mu$ is $X_j$, and because $P_j$ is free as a $kU$-module,
the kernel of $\mu$ is $\Fu^{p-1} (\CY_1 \otimes P_j)$. Because 
$\CY_1 \otimes P_j$ is projective,
we have that 
$$
\Omega(X_j) \cong 
(\CY_1 \otimes P_j)/ \Fu^{p-1}(\CY_1 \otimes P_j).
$$

If $p = 2$, this proves that $\CY_1 \otimes X_j \sim_V X_j$, and we
can iterate the argument to get the conclusion of the lemma.  That is,
$\CY_t \otimes X_j \cong \Omega^t(X_j)$ for all $j$ and all $t$.

So we can assume that $p > 2$. In this situation we observe, by
similar means, that 
$$\Omega^2(X_j) \cong \Fu^{p-1} (\CY_1 \otimes P_j)
\cong \CY_p \otimes X_j.$$ 
So in this case we have that
$$\Omega^{2t}(X_j) \cong \CY_t \otimes X_j.$$
\end{proof}

The final fact we need to complete the proof of Theorem \ref{main-thm}
is the following.

\begin{lemma} \label{lem4} For some $1 \leq i, j \leq t$, suppose that
  $\Hom_{kG}(P_i,P_j)\neq0$.  Then $\homb_{kG}(\CY_k \otimes X_i,
  \CY_{\ell} \otimes X_j)\neq0$ for some $k$ and $\ell$.
\end{lemma}

\begin{proof}
  Let $\xymatrix{\varphi:P_i \ar[r] & P_j}$ be a nonzero homomorphism.
  Let $m$ be the greatest integer such that $\Fu^m\varphi(P_i)\neq0$.
  Then fix a nonzero element $y\in\CY_m$ and define $\xymatrix{ \psi:
    \CY_m \otimes P_i \ar[r] & P_j}$ by $\psi(y \otimes a) = \Fu^m
  \varphi(a)$ for $y \in \CY_m$ and $a \in P_i$. Then $\psi$ is a
  nonzero $kG$-module homomorphism, the kernel of $\psi$ contains
  $\CY_m\otimes\Fu P_i$, and the image of $\psi$ is contained in
  $\Fu^{p-1} P_j \cong \CY_{p-1} \otimes X_j$.  Therefore $\psi$
  induces a nonzero map $\xymatrix{ \psi': \CY_m \otimes X_i \ar[r] &
    \CY_{p-1} \otimes X_j}$.  Finally we need only observe that
  $\psi'$ cannot factor through a projective $kG$-module because its
  restriction to $kU$ does not factor through a projective
  $kU$-module.
\end{proof}

\begin{proof}[Proof of Theorem \ref{main-thm}.]
  As noted, before the proof of Lemma \ref{lem3}, we need only show
  that $X_i \sim_V X_j$ for every $i$ and $j$. Because $P_1, \dots,
  P_t$ are the projective modules in the block $B$, for any $i$ and
  $j$ there is a sequence $i=i_0, \dots, i_r=j$ such that for every $k
  = 1, \dots, r$, $\Hom_{kG}(P_{i_{r-1}}, P_{i_r})\neq0$. So, by Lemma
  \ref{lem3}, there exist $k$ and $\ell$ such that $\CY_k \otimes
  X_{i_{r-1}} \sim_V \CY_{\ell} \otimes X_{i_r}$. The theorem now
  follows by Lemma \ref{lem2}.
\end{proof}

\section{Some examples.}
\label{examples}
In this section we show some examples in which the ext-blocks
corresponding to a subvariety $V \subseteq V_G(k)$ do not coincide
with ordinary blocks.

For the first example let $H$ be an abelian group of order 28
generated by elements $g, x$ and $y$ such that $g^7 = 1$ and $x^2 = 1
= y^2$, and let $G=H\rtimes C_3$ be the semidirect product of $H$ by a
cyclic group $C_3 = \langle z \rangle$ of order 3 acting on $H$ by
$$
zgz^{-1} = g^2 \qquad zxz^{-1} = y \qquad zyz^{-1} = xy.
$$
Let $k$ be a field of characteristic 2 that contains a primitive
$7^{th}$ root of unity, which we denote $\zeta$.  Then $kH$ has seven
simple modules $N_0,\dots,N_6$, each one dimensional, where $g$ acts
on $N_i$ by multiplication by $\zeta^i$. Each simple module is in a
different block of $kH$, and we denote by $b_i$ the block containing
$N_i$.

Then $z$ permutes the simple $kH$-modules, and hence the blocks of
$kH$. That is, $z \otimes N_1 \cong N_4$, $z \otimes N_4 \cong N_2$,
etc. Moreover, $kG$ has exactly three irreducible modules, $k$, $M_1$
and $M_2$ where $(M_1)_H \cong N_1 \oplus N_2 \oplus N_4$ and $(M_2)_H
\cong N_3 \oplus N_5 \oplus N_6$, and each one is the unique simple
module in a block of $kG$. Let $B_0, B_1$ and $B_2$ be the blocks of
$kG$ containing $k, \ M_1$ and $M_2$ respectively.

Now suppose that $V'$ is a line in $V_H(k) \cong k^2$ such that $V'$ 
is not stable under the action of $G/H$, and that 
$V = \res^*_{G,H}(V') \subseteq V_G(k)$.

\begin{prop}
  Each of the subcategories $\CC_V \cap B_1$ and $\CC_V \cap B_2$ is a
  direct sum of three ext-blocks.
\end{prop}

\begin{proof}
  Suppose that $X$ is any module in $B_1$. Then $X_H \cong X_1 \oplus
  X_2 \oplus X_4$, where $X_i$ is a module in $b_i$. Indeed, since
  $\vert G:H \vert$ is not divisible by 2, $X$ is a direct summand of
  $X_H^{\uparrow G}$, and it is easy to see that $X \cong
  X_1^{\uparrow G}$. If, in addition $X$ is an indecomposable object
  of $\CC_V$, then $X_1$ is in exactly one of the subcategories
  $\CC_{V'}$, $\CC_{z(V')}$ or $\CC_{z^2(V')}$. So let $\CU_i$ be the
  subcategory of $B_1 \cap \CC_V$ consisting of all $X$ such that $X_i
  \in \CC_{V'}$ for $i = 1, 2$ or 4.

  Now suppose that $X$ and $Y$ are both in $\CU_i$ for $i = 1,2$ or 4.
  Then $X \cong X_i^{\uparrow G}$ and $Y \cong Y_i^{\uparrow G}$ for
  some $X_i$ and $Y_i$ in $b_i \cap \CC_{V'}$. Hence 
$$ 
\homb_{kG}(X,Y) \cong \homb_{kH}(X_i, Y_1 \oplus Y_2 \oplus Y_4) \neq0,
$$
since
$$
\homb_{kH}(X_i, Y_i) \cong \homb_{kH}(k, X_i^* \otimes Y_i) \neq0,
$$
because $X_i^* \otimes Y_i$ is in the principal block $b_0$.

Suppose on the other hand that $X \in \CU_i$ and $Y \in \CU_j$ for
$i \neq j$. Then, as before,  $X \cong X_i^{\uparrow G}$ and 
$Y \cong Y_j^{\uparrow G}$. In this case
$$ 
\homb_{kG}(X,Y) \cong \homb_{kH}(X_i, Y_1 \oplus Y_2 \oplus Y_4) =0,
$$
since
$$
\homb_{kH}(X_i, Y_i) =0
$$ 
because the varieties of $X_i$ and $Y_i$ intersect trivially, while
 $$
\homb_{kH}(X_i, Y_{\ell}) =0
$$ 
for $\ell \neq i$ because $X_i$ and $Y_{\ell}$ are in different blocks
of $kH$. Hence we have proved that the subcategories $\CU_1$.  $\CU_2$
and $\CU_4$ are the ext-blocks of $\CC_{V} \cap B_1$.
\end{proof}

We should remark that many examples can be constructed along the lines
we have just presented. For example, suppose that $p = 3$ and that $H
= C_5 \times C_3^2$ and $G \cong H \rtimes C_2$, where the generator
of order 2 acts on the $C_5$ and the first $C_3$ by inverting the
elements but acts trivially on the second $C_3$.  Then $H$ has five
blocks, but $G$ has only three. If $V' \subseteq V_H(k)$ is a line
that is not fixed (setwise) by the $C_2$ and if $V = \res_{G,H}^*(V')
\subseteq V_G(k)$, then $B \cap \CC_V$ has two ext-blocks for each
nonprincipal block $B$ of $G$. A similar thing happens in
characteristic five when $G = (C_3 \times C_5^2) \rtimes C_2$, or in
characteristic seven when $G = (C_2^2 \times C_7^2) \rtimes C_3$.  We
shall show that these examples are typical of what happens in general.

The examples also show some unusual behavior of the idempotent modules.
Specifically, we have the following. Assume the hypothesis and notation
of the example at the beginning of the section.

\begin{cor}
  Let $G = (C_7 \times C_2^2) \rtimes C_3$ and $V$ be as in the example.
  Suppose that $e_V$ is the idempotent module corresponding to the
  subvariety $V$. Recall that $M_1$ is the unique simple module in the
  block $B_1$. Then $e_V \otimes M_1$ is a sum of three modules, one in
  each ext-block.
\end{cor}

\begin{proof}
  Suppose that $X$ is a nonprojective module in $\CC_V \cap B_1$.
  Then, because $M_1$ is the unique simple module in $B_1$, we must
  have that $\homb_{kG}(M_1,X)\neq0$ and hence also that $\homb_{kG}(e_V
  \otimes M_1, X) \neq0$. It follows that $e_V \otimes M_1$ must have
  a component in every ext-block of $\CC_V \cap B_1$.
\end{proof}

\section{Lines in general} 
\label{cb} 

In this section, we reduce the study of ext-blocks in $\CC_V$ for an
arbitrary line $V$ to the case studied in Section~\ref{sec-fixed}
using the following theorem of Benson~\cite{ben-catthm}.

\begin{thm} \label{benson-thm}
\cite{ben-catthm} Suppose that $V$ is a line in $V_G(k)$ which is 
minimally supported on an elementary abelian subgroup $E$. Suppose 
that $V'\subseteq V_E(k)$ is a line such that $\res^*_{G,E}(V') = V$.
Let $H$ be the set-wise stabilizer of $V'$ in $N_G(E)$, and let
$\hat{V} = \res^*_{H,E}(V')$. Then the categories $\CC_V$ and 
$\CC_{\hat{V}}$ are equivalent. 
\end{thm}

The equivalence of categories is easy to describe.  The functor
$\xymatrix{\CC_{\hat{V}} \ar[r] & \CC_V}$ is simply induction from $H$
to $G$, and the inverse functor $\xymatrix{\CC_V \ar[r] &
  \CC_{\hat{V}}}$ is restriction to $H$ followed by choosing the
largest direct summand of the restriction that has variety $\hat{V}$.
This last operation is equivalent to taking the tensor product with
$e_{\hat{V}}$, the idempotent module corresponding to $\hat{V}$. The
key point of the proof that these functors are equivalences of
categories is that the conditions guarantee that in the Mackey
decomposition
$$
(M^{\uparrow G})_H \ \cong \ \sum_{HxH} \ x \otimes 
(M_{H \cap xHx^{-1}})^{\uparrow H},
$$
the terms $(M_{H \cap xHx^{-1}})^{\uparrow H}$ for $x\notin H$ have
varieties that intersect $\hat{V}$ trivially, so that if
$M\in\CC_{\hat{V}}$ then the largest direct summand of $(M^{\uparrow
  G})_H$ with variety $\hat{V}$ is stably isomorphic to $M$. See
\cite{ben-catthm} for details.

Benson's equivalence and Theorem~\ref{main-thm} easily imply the
following.

\begin{prop} \label{line-blocks}
Let $V$ be a line in $V_G(k)$ minimally supported on $E$, and let $H$
and $\hat{V}$ be as above. Then the ext-blocks of $\CC_V$ are
parametrized by the ordinary blocks of $kH$.
\end{prop}

\begin{proof}
By Benson's theorem, there is a natural bijection between ext-blocks
of $\CC_V$ and of $\CC_{\hat{V}}$. But Theorem~\ref{main-thm} applies
to $\CC_{\hat{V}}$, so the ext-blocks are given by the ordinary blocks
of $kH$.
\end{proof}

Of course, there is also a direct sum decomposition of $\CC_V$ given
by the blocks of $kG$, and so each ext-block is contained in an
ordinary block. The way this happens is controlled by Brauer
correspondence. Note that since 
$$C_G(E)=EC_G(E)\leq H\leq N_G(E),$$
for each block $b$ of $kH$ there is a unique block $b^G$ of $kG$, the
Brauer correspondent of $b$, and the blocks of $kG$ that occur in this
way are those with defect groups containing $E$.

\begin{prop}
  In the situation of Proposition~\ref{line-blocks}, every
  non-projective indecomposable module $M$ in the
  ext-block of $\CC_V$ corresponding to a block $b$ of $kH$ is in the
  ordinary block $b^G$ of $kG$. 
\end{prop}

\begin{proof}
  This follows from Nagao's module-theoretic form of Brauer's Second
  Main Theorem, which tells us that if $M$ is in the block $B$ of
  $kG$, then $M_H=M'\oplus M''$, where $M'$ is a direct sum of modules
  in blocks of $kH$ which have $B$ as their Brauer correspondent and
  $M''$ is a direct sum of modules projective relative to subgroups of
  $H$ which do not contain $E$. But then $\hat{V}$ is not contained in
  the variety of $M''$, so the image of $M$ under Benson's equivalence
  must be a summand of $M'$.
\end{proof}

\section{Linckelmann's block varieties} 
\label{linck} 

In this section, we shall consider how the previous results are
related to Linckelmann's notion of block varieties~\cite{L1}.

Let us briefly recall the definition. Let $B$ be a block of $kG$ with
defect group $D$. Choose a maximal $B$-Brauer pair $(D,e_D)$, and for
each $Q\leq D$ let $e_Q$ be the unique block idempotent of $kC_G(Q)$
such that $(Q,e_Q)\leq(D,e_D)$. So in particular $e_{\{1\}}$ is the
block idempotent corresponding to the block $B$.

Let $\CF_{G,B}$ be the fusion system of the block: i.e., the category
whose objects are subgroups of $D$ and where a morphism from $Q$ to
$R$ is a group homomorphism induced by conjugation by some $x\in G$
such that $^x(Q,e_Q)\leq(R,e_R)$. Then Linckelmann defines the block
cohomology $H^*(G,B)$ to be $\varprojlim H^*(Q,k)$, where the inverse
limit is over the category $\CF_{G,B}$. More concretely, $H^*(G,B)$ is
the subring of $H^*(D,k)$ consisting of elements that are stable in a
suitable sense. Then the variety $V_{G,B}$ is the maximal ideal
spectrum of $H^*(G,B)$.

The inclusion $H^*(G,B)\rightarrow H^*(D,k)$, composed with
restriction $H^*(D,k)\rightarrow H^*(Q,k)$ induces a map of varieties
$$r^*_Q:V_Q(k)\rightarrow V_{G,B}$$
for each subgroup $Q\leq D$, and in particular
$$r^*_D:V_D(k)\rightarrow V_{G,B}$$ 
is a finite surjective map.

Also, the image of the restriction map $H^*(G,k)\rightarrow H^*(D,k)$
is contained in $H^*(G,B)$, so there is a natural map of varieties
$$\rho_B:V_{G,B}\rightarrow V_G(k).$$

Linckelmann also defines a subvariety $V_{G,B}(M)$ of $V_{G,B}$, the
{\em block variety}, for every finitely generated module $M$ in the
block $B$ in such a way that $\rho_B$ induces a finite surjective map
$$\rho_B:V_{G,B}(M)\to V_G(M).$$ 
So, as an invariant of the module $M$, $V_{G,B}(M)$ may be regarded as
a refinement of $V_G(M)$.

We shall show that Linckelmann's varieties give another way of
constructing direct sum decompositions of the categories
$\CC_V$. First, we need to generalize some familiar properties of
varieties for modules to this setting.

We shall use the following useful theorem of Linckelmann~\cite[Theorem
2.1]{L3}.

\begin{thm} 
\label{source}
  Let $B$ be a block of $G$ with defect group $D$, and let $i$ be a
  source idempotent of $B$. Then for any finitely generated module $M$
  in the block $B$,
$$V_{G,B}(M)=r^*_D\left(V_D(iM)\right),$$
where $iM$ is considered as a $kD$-module.
\end{thm} 

\begin{lemma}
\label{2outof3}
Let 
$$0\to M_1\to M_2 \to M_3\to 0$$
be a short exact sequence of modules in the block $B$. Then 
$$V_{G,B}(M_{\alpha})\subseteq V_{G,B}(M_{\beta})\cup V_{G,B}(M_{\gamma})$$
for $\{\alpha,\beta,\gamma\}=\{1,2,3\}$.
\end{lemma}

\begin{proof}
  Using Theorem~\ref{source} this follows easily from the well-known
  corresponding statement for cohomological varieties. Multiplying the short exact sequence by the source idempotent $i$, we get a short exact sequence
$$0\to iM_1\to iM_2\to iM_3\to 0,$$
and so
\begin{align*}
V_{G,B}(M_{\alpha})&=r^*_D\left(V_D(iM_{\alpha})\right)\\
&\subseteq r^*_D\left(V_D(iM_{\beta})\cup V_D(iM_{\gamma})\right)\\
&= r^*_D\left(V_D(iM_{\beta})\right)\cup 
r^*_D\left(V_D(iM_{\gamma})\right)\\
&=V_{G,B}(M_{\beta})\cup V_{G,B}(M_{\gamma}).
\end{align*}
\end{proof}

The following follows immediately for the stable module category.
 
\begin{cor}
\label{thick}
  Let $W$ be a closed homogeneous subvariety of $V_{G,B}$. Then the
  finitely-generated modules $M$ in the block $B$ for which
  $V_{G,B}(M)\subseteq W$ form a thick subcategory of $\stmodg$.
\end{cor}

We shall denote this thick subcategory by $\CC_{W,B}$.

\begin{prop} 
  Let $M$ and $N$ be finitely generated modules in a block $B$ of a
  finite group $G$. If $V_{G,B}(M)\cap V_{G,B}(N)=\{0\}$, then
  $\homb(M,N)=0$.
\end{prop}

\begin{proof}
  Suppose $\phi:M\to N$ is a homomorphism between modules whose
  varieties intersect trivially. We can complete this map to a
  triangle
$$M\to N\to L\to \Omega(M)$$
in $\stmodg$. By Corollary~\ref{thick},
$$V_{G,B}(L)=V_{G,B}(M)\cup V_{G,B}(N).$$
 
In~\cite[Corollary 1.2]{BL}, Benson and Linckelmann prove that the
block variety of an indecomposable module is connected, and
by~\cite[Corollary 2.2]{L3}, the block variety of a direct sum of two
modules is the union of their individual block varieties. It follows
that $L$ is the direct sum of two modules $L_M$ and $L_N$, with
$V_{G,B}(L_M)=V_{G,B}(M)$ and $V_{G,B}(L_N)=V_{G,B}(N)$.

The octahedral axiom gives a commutative diagram
$$\xymatrix{
\Omega(L)\ar[r]\ar[d]&M\ar[r]\ar[d]&N\ar @{=}[d]\\
\Omega(L_N)\ar[r]\ar[d]&X\ar[r]\ar[d]&N\\
L_M\ar @{=}[r]&L_M
}$$
where, by Corollary~\ref{thick}, 
$$V_{G,B}(X)\subseteq V_{G,B}(M)\cap V_{G,B}(N)=\{0\},$$
and so $X$ is projective. But the map $\phi$ factors through $X$.
\end{proof}

The next corollary follows immediately.

\begin{cor}
  If $W=\cup_{i\in I}W_i$ is the union of finitely many closed
  subvarieties $W_i$, where $W_i\cap W_j=\{0\}$ for $i\neq j$, then
  $\CC_{W,B}$ has a direct sum decomposition
$$\CC_{W,B}=\bigoplus_{i\in I}\CC_{W_i,B}.$$
\end{cor}

Now let us return to the example of $G=(C_7\times C_2^2)\rtimes C_3$
studied in Section~\ref{examples}. Recall that in that example, there
was a line $V$ in $V_G(k)$, and two blocks $B_1$ and $B_2$ for which
the intersection of $\CC_V$ with each block decomposed as the direct
sum of three ext-blocks.

Considering Linckelmann's block varieties sheds new light on this. Let
$B$ be either of the two blocks, and recall that there is a natural
map of varieties $\rho_B:V_{G,B}\to V_G(k)$. In this example, one can
calculate that $\rho_B^{-1}(V)$ is the union of three lines
$W_1,W_2,W_3$ in $V_{G,B}$, and the intersection of $\CC_V$ with the
block $B$ is the direct sum
$\CC_{W_1,B}\oplus\CC_{W_2,B}\oplus\CC_{W_3,B}$ of thick subcategories
determined by block varieties.

Similar observations apply in all other examples we have calculated,
and it is natural to ask whether it is true in general, given a block
$B$ of a finite group $G$ and a line $V$ in the image of $\rho_B$,
that the ext-blocks of the intersection of $\CC_V$ with the block $B$
are precisely the categories $\CC_{W,B}$, for $W$ an irreducible
component of $\rho_B^{-1}(V)$. If this were the case, then it would
follow by a fairly straightforward argument that for any closed
homogeneous subvariety $V$ of $V_G(k)$, the ext-blocks of the
intersection of $\CC_V$ with the block $B$ are just the categories
$\CC_{W,B}$ for $W$ a connected component of $\rho_B^{-1}(V)$.

\end{document}